\documentclass[12pt,a4paper]{article}

\usepackage{amsmath,amssymb,amsthm}
\usepackage[margin=27mm]{geometry}
\usepackage[hidelinks]{hyperref}
\usepackage{enumitem}

\newtheorem{theorem}{Theorem}[section]
\newtheorem{lemma}[theorem]{Lemma}
\newtheorem{proposition}[theorem]{Proposition}
\newtheorem{corollary}[theorem]{Corollary}
\theoremstyle{definition}
\newtheorem{definition}[theorem]{Definition}

\newtheorem{problem}[theorem]{Problem}
\theoremstyle{remark}

\newcommand{\zr}{z_{\mathrm{RL}}}
\newcommand{\zw}{z_{\mathrm{wL}}}
\newcommand{\rw}{\mathrm{(RW3^{+})}}
\newcommand{\RR}{\mathbb{R}}
\newcommand{\supp}{\operatorname{supp}}
\newcommand{\sos}{\operatorname{SOS}}
\newcommand{\bsr}{\operatorname{BSR}}
\newcommand{\ip}[2]{\langle #1,#2\rangle}
\newcommand{\floor}[1]{\left\lfloor #1\right\rfloor}
\newcommand{\doi}[1]{\href{https://doi.org/#1}{doi:\nolinkurl{#1}}}
\newcommand{\arxiv}[1]{\href{https://arxiv.org/abs/#1}{arXiv:\nolinkurl{#1}}}

\begin{document}

\title{Recursive-Line Zarankiewicz Numbers with Four Columns}
\author{Zhiwei Chen\footnote{School of Mathematical Sciences, South China Normal University, Guangzhou 510631, China (\texttt{email:~javenchen2002@163.com}).} \and
  Yannan Chen\footnote{School of Mathematical Sciences, South China Normal University, Guangzhou 510631, China (\texttt{email:~ynchen@scnu.edu.cn}). This author was supported by the National Natural Science Foundation of China (12671423, 12471351).}}
\date{\today}
\maketitle

\begin{abstract}
  The recursive-line Zarankiewicz number is defined on augmentations of extremal $C_4$-free bipartite graphs. It maximizes the total number of selected one-edges and two-edges whose associated sum-of-squares representation has a recursive irreducibility certificate. We study the four-column case under the strengthened definition of L\"ofberg and Qi. Exact verification and finite exclusion give a table for $2\le m\le20$: eighteen values are exact, while
  $37\le\zr(14,4)\le38$. Our main result is
  \[
    \zr(m,4)=\floor{\frac{5m+6}{2}}\qquad(m\ge15).
  \]
  The upper bound follows from the classical identity $z(m,4)=m+6$ and a cell count. For the lower bound, one $16\times4$ seed generates every even order by extending two parallel chains. Deleting a fixed row gives every odd order starting at fifteen. We prove admissibility by a grounded path-transport argument and a column-wise induction. The appendix gives a complete elementary certificate for the seed. The same formula holds for the second-order number $z_2(m,4)$. For $m\ge15$, every extremal configuration has no holes when $m$ is even and exactly one hole when $m$ is odd.


\end{abstract}

\newpage

\section{Introduction}\label{sec:intro}

The recursive-line Zarankiewicz problem is an optimization problem at the interface of extremal combinatorics, polynomial optimization, and semidefinite modeling. An augmented bipartite configuration encodes a sparse sum-of-squares (SOS) decomposition of a biquadratic polynomial: one-edges represent monomial squares, whereas two-edges represent squares of two-term bilinear forms. Thus the quantity being maximized is simultaneously a combinatorial packing objective and the size of a structured SOS/Gram decomposition subject to coefficient-matching, disjoint-support, and recursive-rank constraints. This perspective is consistent with recent work on semidefinite complementarity, trust-region global optimization, and tensor-norm regularization \cite{GaoQuCui25,LiHaoXuGaoLi25,ZhaoZhouFan25}. In our setting, the optimization viewpoint is not merely motivational: it guides the exact finite searches, the symmetry reductions, the extremal upper bound, and the certificate-replay arguments that establish the infinite four-column formula.

The classical Zarankiewicz problem asks for the largest number of edges in a bipartite graph that contains no $K_{2,2}$, or equivalently no $4$-cycle. In grid form, it asks for the largest number of occupied cells in an $m\times n$ array with no fully occupied $2\times2$ subarray. We denote this number by $z(m,n)$. The basic theory goes back to Zarankiewicz \cite{Zar51}, K\H{o}v\'ari, S\'os and Tur\'an \cite{KST54}, \v{C}ul\'ik \cite{Culik56}, and Reiman \cite{Reiman58}. Later work improved finite and asymptotic bounds; see, for example, \cite{Roman75,Furedi96,FS13,CHM24,DGH26}. When the number of columns is fixed, the extremal graphs have strong structure. In particular, \v{C}ul\'ik's result \cite{Culik56} gives the four-column skeleton used in this paper.

The augmented problem is motivated by sums of squares of biquadratic forms. A one-edge represents a monomial square, while a two-edge represents the square of a two-term bilinear form. The selected cells therefore encode a sparse SOS representation. The goal is to use as many displayed squares as possible while keeping the representation irreducible. This connects the combinatorial problem with coefficient identities and Gram-rank questions. Background on polynomial sums of squares can be found in \cite{Choi75,CLR95}. The augmented-graph model and several related constructions were developed in \cite{QCXrank26,QCXaug26,QCXweak26,QCXinc26}.

The parameter studied here is based on a recursive certificate for these coefficient identities. Qi, L\"ofberg and Chen \cite{QLC26} studied weak limited augmented numbers in the three-column case. L\"ofberg and Qi \cite{LQrevision26} introduced the second-order number $z_2$ and the recursive-line number $\zr$. Their strengthened closure replaces a literal empty-cross-cell condition by finitely many line and rectangle rules. In the formulation used here, they prove
\begin{equation}\label{eq:hierarchy}
\bsr(m,n)\ge z_2(m,n)\ge\zr(m,n)\ge\zw(m,n)\ge z(m,n),
\end{equation}
and the exact three-column identity
\begin{equation}\label{eq:n=2}
  \zr(m,3)=z_2(m,3)=2m, \qquad (m\ge3),
\end{equation}
as well as several small four-column values. The eventual behavior for four columns remains open in that manuscript.

Throughout the paper we use \cite[Definition~4.11]{LQrevision26}: the configuration is simple, its one-edge graph is $C_4$-free, and it satisfies strengthened recursive admissibility. 
This paper has two main contributions. First, we give certified data for $2\le m\le20$, including all values for $8\le m\le13$ and the bounds $37\le\zr(14,4)\le38$. Second, the zero-or-one-hole pattern for $15\le m\le20$ leads to an explicit infinite family and the following theorem.

\begin{theorem}\label{thm:main}
Under Definition~\ref{def:zrl}, for every integer $m\ge15$,
\begin{equation}\label{eq:main}
\zr(m,4)=z_2(m,4)=\floor{\frac{5m+6}{2}}.
\end{equation}
Every extremal configuration for $\zr(m,4)$ has $H=0$ holes when $m$ is even and $H=1$ hole when $m$ is odd.
\end{theorem}

The lower-bound construction extends two parallel chains. Each step adds two rows, two one-edges, and three net two-edges. Deleting one fixed row gives the odd-order family, starting at order fifteen. The proof combines a distance-decreasing transport lemma, a finite seed certificate, and a column-wise induction. Exhaustive computation is used only for the small-order exclusions. Equality with $z_2$ follows from the hierarchy and the same cell bound; it does not determine $\bsr(m,4)$.

Section~\ref{sec:def} gives the definitions, the closure rules, and a deletion lemma. Section~\ref{sec:finite} presents the finite computation, the table for $2\le m\le20$, and sample witnesses. Section~\ref{sec:general} proves the upper bound and constructs the infinite family. Section~\ref{sec:conclusion} discusses $m=14$, the stabilization threshold, and further questions.

\section{Recursive-line admissibility}\label{sec:def}

This section recalls the augmented-grid framework and the recursive-line admissibility criterion from L\"ofberg and Qi \cite[Definitions~4.1, 4.2, 4.5 and 4.11]{LQrevision26}. The terminology and definitions below are included for completeness.

For a positive integer $r$, let $[r]=\{1,\ldots,r\}$. Identify the edges of the complete bipartite graph with parts $[m]$ and $[n]$ with the cells of $[m]\times[n]$. A set $E_1$ of \emph{one-edges} is therefore an ordinary bipartite graph. A \emph{two-edge} is an unordered pair $\{p,q\}$ of distinct cells. If $p=(i,j)$ and $q=(k,\ell)$, it is row-degenerate when $i=k$, column-degenerate when $j=\ell$, and nondegenerate when $i\ne k$ and $j\ne\ell$. Two nondegenerate two-edges are complementary when they are the two diagonals of the same genuine rectangle.

For a one-edge $e=(i,j)$, set $\supp(e)=\{(i,j)\}$; for a two-edge $e=\{p,q\}$, set $\supp(e)=\{p,q\}$. An augmented configuration is $G=([m],[n],E_1\cup E_2)$. Its simplicity condition (S) says that supports of distinct selected edges are disjoint. Thus a cell cannot be used twice, although different selected edges may meet the same row or column vertex. Write
\[
\Omega=E_1\cup\bigcup_{e\in E_2}e,
\qquad H=mn-|\Omega|,
\qquad N(G)=|E_1|+|E_2|.
\]
A two-edge uses two cells but contributes only one to $N(G)$. The configuration is \emph{limited} if $|E_1|=z(m,n)$. This is an equality requirement: one may not replace one-edges by additional two-edges while leaving the same parameter unchanged.

The one-edge graph is $C_4$-free precisely when no genuine rectangle has all four corners in $E_1$. This does not prohibit a rectangle whose corners belong to two-edges. We depict each one-edge by $\bullet$, each hole by $\circ$, and each two-edge by a letter or integer occurring exactly twice. Different black dots always denote different one-edges.

\begin{figure}[ht]
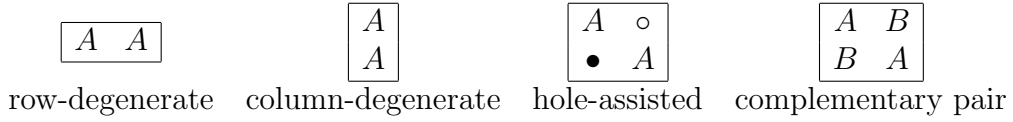

\centering
\begin{tabular}{cccc}
$\begin{array}{|cc|}\hline A&A\\\hline\end{array}$&
$\begin{array}{|c|}\hline A\\A\\\hline\end{array}$&
$\begin{array}{|cc|}\hline A&\circ\\\bullet&A\\\hline\end{array}$&
$\begin{array}{|cc|}\hline A&B\\B&A\\\hline\end{array}$\\[5pt]
row-degenerate&column-degenerate&hole-assisted&complementary pair
\end{tabular}
\caption{Local grid patterns illustrating resolution of selected two-edges. The first two use the line rule. The third uses rectangle transfer from a diagonal containing a hole. In the last pattern, the two $A$-cells are identified and the two $B$-cells are identified, but the two edges are not identified with each other. The patterns are local illustrations, not assertions that each displayed subgrid is itself limited.}\label{fig:rules}
\end{figure}

For distinct cells $p,q$, define
\[
\delta(p,q)=\begin{cases}1,&\{p,q\}\in E_2,\\0,&\text{otherwise}.\end{cases}
\]
A genuine rectangle has diagonals $\{p,q\}$ and $\{r,s\}$, where
\[
p=(i,j),\quad q=(k,\ell),\quad r=(i,\ell),\quad s=(k,j),
\qquad i\ne k,\quad j\ne\ell.
\]
Following \cite[Definitions 4.1 and 4.5]{LQrevision26}, construct the least equivalence relation $\sim$ on $\Omega$ and the least symmetric certified-orthogonality relation $\perp_R$ closed under the following rules.
\begin{itemize}
  \item \textbf{Line rule.} For two occupied cells on the same row or column, certify their prescribed value: if $\delta(p,q)=0$, add $p\perp_R q$; if $\delta(p,q)=1$, add $p\sim q$.
  \item \textbf{Saturation rule.} If $p\sim p'$, $q\sim q'$, and $p'\perp_R q'$, then add $p\perp_R q$.
  \item \textbf{Rectangle transfer.} Suppose the companion diagonal $\{r,s\}$ is already certified to have its prescribed value. For value zero, this means that at least one of its cells is a hole, or that $r\perp_R s$. For value one, both cells must be occupied and $r\sim s$. Then certify the target diagonal: if $p,q$ are occupied and $\delta(p,q)=0$, add $p\perp_R q$; if $\delta(p,q)=1$, add $p\sim q$.
  \item \textbf{Complementary-pair rule.} If both diagonals of a genuine rectangle are selected two-edges $\{p,q\}$ and $\{r,s\}$, identify the two halves of each diagonal simultaneously: $p\sim q$ and $r\sim s$.
\end{itemize}
The closure is a least fixed point: only consequences reached by finitely many applications of these rules are certified. In particular, an algebraic argument involving an ungrounded cycle of equations is not automatically a valid closure certificate.

\begin{definition}[Strengthened recursive admissibility, see {\cite[Definitions 4.2 and 4.5]{LQrevision26}}]\label{def:rw}
A simple augmented configuration satisfies $\rw$ if its strengthened closure has the following properties: every selected two-edge has its two halves identified; distinct selected edges determine distinct equivalence classes; and, for any distinct selected edges $e,f$, there exist $p\in\supp(e)$ and $q\in\supp(f)$ such that $p\perp_R q$.
\end{definition}

\begin{definition}[Recursive-line Zarankiewicz number]\label{def:zrl}
In the sense of \cite[Definition 4.11]{LQrevision26},
\begin{equation}\label{eq:def-zrl}
\zr(m,n)=\max_G\left\{|E_1|+|E_2|:\
\begin{array}{l}
|E_1|=z(m,n),\ G\text{ satisfies (S)},\\
G_1\text{ is }C_4\text{-free},\ G\text{ satisfies }\rw
\end{array}\right\}.
\end{equation}
Here $G_1=([m],[n],E_1)$, and row- and column-degenerate two-edges are allowed. No additional dependency-acyclicity or opposite-one-edge restriction is imposed.
\end{definition}

Associate with $G$ the doubly simple biquadratic form
\begin{equation}\label{eq:polynomial}
P_G(x,y)=\sum_{(i,j)\in E_1}(x_i y_j)^2+
\sum_{\{(i,j),(k,\ell)\}\in E_2}(x_i y_j+x_k y_\ell)^2.
\end{equation}
The SOS rank $\sos(P_G)$ is the minimum number of squares of real bilinear forms in an SOS representation. In any such representation with $r$ squares, let $v_p\in\RR^r$ be the vector of coefficients of the monomial associated with cell $p$. Coefficient comparison gives
\begin{align}
\|v_p\|^2&=1 \quad(p\in\Omega),& &v_p=0 \quad(p\notin\Omega),\label{eq:unit}\\
\ip{v_p}{v_q}&=\delta(p,q) &&\text{for distinct cells on a line},\label{eq:line}\\
\ip{v_p}{v_q}+\ip{v_r}{v_s}&=\delta(p,q)+\delta(r,s)
&&\text{for the two diagonals of a rectangle}.\label{eq:rectangle}
\end{align}
For a complementary selected pair, the right-hand side of \eqref{eq:rectangle} is two. Each inner product is at most one, so both equal one and both selected edges resolve. This explains the last pattern of Figure~\ref{fig:rules}.

\begin{proposition}[Soundness; \cite{LQrevision26}]\label{prop:sound}
If $G$ satisfies $\rw$, then $$\sos(P_G)=N(G).$$
\end{proposition}


The second-order number $z_2(m,n)$ \cite[Definition 3.1]{LQrevision26} is the maximum $N(G)$ over simple, limited configurations with $C_4$-free one-edge graph for which \eqref{eq:polynomial} is irreducible, whether or not the chosen closure certifies it. Proposition~\ref{prop:sound} gives $z_2\ge\zr$.

\begin{lemma}[Deletion of complete selected edges]\label{lem:delete}
If an augmented configuration satisfies $\rw$, deleting any collection of complete selected edges, together with all cells in their supports, preserves $\rw$ on the retained configuration. This statement concerns admissibility; the limited equality must be checked separately. An entirely empty row or column can subsequently be removed without affecting admissibility.
\end{lemma}
\begin{proof}
Retained occupied cells have the same selected-edge membership and the same pair values $\delta$ as before. Reproduce the original derivation in order, retaining the conclusions concerning retained cells. Every identification joins the two halves of one selected two-edge, so no retained equivalence class needs a deleted cell as an intermediate representative. A line or complementary-pair step on retained cells remains valid. In a rectangle step, either the companion cells are both retained, in which case their previously certified value is reproduced, or one is deleted, in which case the new companion is automatically certified zero. The target is still certified to its prescribed value by the rectangle rule. Saturation is therefore preserved as well. The final identifications and orthogonality obligations for retained edges all survive. The word ``orthogonal'' consequently concerns coefficient vectors, not the angle between drawn graph edges. 

If a row is empty, a rectangle containing that row cannot have both target corners occupied: each diagonal meets the empty row. Such a rectangle supplies no necessary conclusion about a pair of retained occupied cells. The empty row may thus be removed and the others relabelled. The column statement is identical.
\end{proof}

\section{Finite computations and exact verification}\label{sec:finite}

Computational search was used to find candidate configurations and extension patterns. Every reported lower bound is checked by the exact closure rules of Section~\ref{sec:def}. One program records a finite derivation, and an independent replay checker verifies every rule application and every final obligation. All data are discrete, so no floating-point tolerance is used.

A positive certificate proves that one explicit configuration is admissible, and hence gives a lower bound. It gives an exact value only when a matching upper bound is also known. For $m=13$ and $15\le m\le20$, the elementary cell bound in Section~\ref{sec:general} supplies this upper bound. For $8\le m\le12$, we instead use exhaustive finite exclusions. 

For $m\ge6$, Lemma~\ref{lem:skeleton} reduces all one-edge extremal graphs to six fixed pair-rows and $m-6$ singleton rows. Up to row and column permutations, their types are parametrized by
\begin{equation}\label{eq:partitions}
0\le a_1\le a_2\le a_3\le a_4,
\qquad a_1+a_2+a_3+a_4=m-6,
\end{equation}
where $a_j$ counts the singleton rows whose one-edge is in column $j$. For each type, the search enumerates disjoint two-edge families on the remaining cells, allowing both kinds of degenerate edges and complementary pairs. On $f$ free cells, the unpruned number of families with $k$ pairs is
\begin{equation}\label{eq:matchings}
M(f,k)=\frac{f!}{(f-2k)!\,2^k k!}.
\end{equation}
This counts labelled pair families before quotienting by further automorphisms; it does not count isomorphism classes.

The search branches on an unassigned free cell. The cell is either left as a hole or paired with a later free cell. Lemma~\ref{lem:delete} gives a hereditary pruning rule. Treat the unassigned cells as holes. If the chosen partial family fails the full closure test, then no admissible completion can contain it; otherwise deleting the future edges from such a completion would contradict the lemma. The archived implementation also uses pair and triple compatibility, necessary matching conditions, and permutations of identical untouched singleton rows. 

Table~\ref{tab:finite} gives eighteen exact values of $\zr(m,4)$ for $2\le m\le20$. The remaining value satisfies
\begin{equation}\label{eq:m14}
37\le\zr(14,4)\le38.
\end{equation}
The quantities $|E_2|$ and $H$ in the table describe the supplied witness; in the row $m=14$, that witness is only a lower-bound construction.

\begin{table}[ht]
\centering
\small
\setlength{\tabcolsep}{7pt}
\renewcommand{\arraystretch}{1.13}
\begin{tabular}{rcrrrrrl}
\hline
$m$&$\zr(m,4)$&$|E_1|$&$|E_2|$&$H$&$U_m$&$\binom{N(G)}2$&Basis\\\hline
2&5&5&0&3&6&10&P\\
3&8&7&1&3&9&28&P\\
4&10&9&1&5&12&45&P\\
5&13&10&3&4&15&78&P\\
6&16&12&4&4&18&120&P\\
7&19&13&6&3&20&171&P\\
8&21&14&7&4&23&210&E\\
9&24&15&9&3&25&276&E\\
10&27&16&11&2&28&351&E\\
11&29&17&12&3&30&406&E\\
12&32&18&14&2&33&496&E\\
13&35&19&16&1&35&595&C\\
14&$[37,38]$&20&17&2&38&666&Lower bound\\
15&40&21&19&1&40&780&C\\
16&43&22&21&0&43&903&C\\
17&45&23&22&1&45&990&C\\
18&48&24&24&0&48&1128&C\\
19&50&25&25&1&50&1225&C\\
20&53&26&27&0&53&1378&C\\\hline
\end{tabular}
\caption{Certified data. Here $U_m=\lfloor(4m+z(m,4))/2\rfloor$ is the universal cell bound. P denotes a value inherited from \cite{LQrevision26}, possibly by transposition; E denotes complete finite exclusion above the supplied witness; C denotes attainment of the cell bound. Every listed witness has all $\binom{N(G)}2$ distinct-edge orthogonality obligations certified.}\label{tab:finite}
\end{table}

The exact two- and three-column results of \cite{LQrevision26} give $\zr(2,4)=5$ and $\zr(3,4)=8$. The same source gives the entries for $m=4,5,6,7$. 

For $8\le m\le12$, Table~\ref{tab:exclusion} specifies the complete exclusion tasks. Each listed one-edge type was exhausted and no admissible configuration at the stated target size was found. Configurations of larger size are excluded as well: deleting complete two-edges would produce a configuration at the forbidden target size with the same one-edge set. Together with the witnesses, this proves the five exact values. These upper-bound computations are the archived exhaustive runs.

For $m=13$ and $15\le m\le20$, the witness size equals $U_m$, establishing equality by the universal upper bound. At $m=14$, the verified witness has $20+17=37$ edges, whereas $U_{14}=38$. The evidence establishes only the interval \eqref{eq:m14}.

\begin{table}[ht]
\centering
\begin{tabular}{rrrrrl}
\hline
$m$&Target $N$&Target $|E_2|$&Target $H$&Skeleton types&Result\\\hline
8&22&8&2&2&All excluded\\
9&25&10&1&3&All excluded\\
10&28&12&0&5&All excluded\\
11&30&13&1&6&All excluded\\
12&33&15&0&9&All excluded\\\hline
\end{tabular}
\caption{Complete upper-bound exclusions supporting the exact values for $8\le m\le12$ in Table~\ref{tab:finite}. Types are the partitions in \eqref{eq:partitions}. The number of types is not the number of pair families or backtracking states. }\label{tab:exclusion}
\end{table}

The four configurations in Figure~\ref{fig:finite-bases} illustrate the finite discovery stage. Their totals are $40,43,45,48$ and their hole counts are $1,0,1,0$. The proof in Section~\ref{sec:general} will use only the sixteen-row witness, after a column permutation; all larger even orders and all odd orders in the theorem follow from that single seed. Figure~\ref{fig:seed} displays the normalized seed and its regular twenty-row extension. The odd witnesses obtained by deletion need not be isomorphic to the historical witnesses shown here.

\begin{figure}[ht]
\centering
\begin{minipage}[t]{.24\textwidth}\centering
\textbf{$m=15$}\\[5pt]
{\footnotesize\setlength{\arraycolsep}{2.4pt}\renewcommand{\arraystretch}{1.12}$\begin{array}{r|rrrr}
 &1&2&3&4\\\hline
1&\bullet&\bullet&1&2\\
2&\bullet&3&\bullet&4\\
3&\bullet&5&5&\bullet\\
4&6&\bullet&\bullet&7\\
5&8&\bullet&2&\bullet\\
6&3&9&\bullet&\bullet\\\hline
7&7&10&8&\bullet\\
8&\circ&11&\bullet&12\\
9&13&\bullet&14&15\\
10&16&10&\bullet&16\\
11&11&6&\bullet&9\\
12&1&\bullet&13&12\\
13&4&\bullet&15&14\\
14&17&17&\bullet&18\\
15&18&19&\bullet&19\\
\end{array}$}\\[5pt]
$N=40,\ H=1$.
\end{minipage}\hfill
\begin{minipage}[t]{.24\textwidth}\centering
\textbf{$m=16$}\\[5pt]
{\footnotesize\setlength{\arraycolsep}{2.4pt}\renewcommand{\arraystretch}{1.12}$\begin{array}{r|rrrr}
 &1&2&3&4\\\hline
1&\bullet&\bullet&1&1\\
2&\bullet&2&\bullet&3\\
3&\bullet&4&5&\bullet\\
4&6&\bullet&\bullet&7\\
5&8&\bullet&9&\bullet\\
6&10&6&\bullet&\bullet\\\hline
7&11&12&\bullet&13\\
8&13&7&\bullet&11\\
9&14&15&16&\bullet\\
10&16&17&14&\bullet\\
11&15&18&10&\bullet\\
12&19&19&17&\bullet\\
13&12&3&8&\bullet\\
14&20&21&21&\bullet\\
15&18&5&20&\bullet\\
16&2&9&4&\bullet\\
\end{array}$}\\[5pt]
$N=43,\ H=0$.
\end{minipage}\hfill
\begin{minipage}[t]{.24\textwidth}\centering
\textbf{$m=17$}\\[5pt]
{\footnotesize\setlength{\arraycolsep}{2.4pt}\renewcommand{\arraystretch}{1.12}$\begin{array}{r|rrrr}
 &1&2&3&4\\\hline
1&\bullet&\bullet&1&2\\
2&\bullet&3&\bullet&4\\
3&\bullet&5&2&\bullet\\
4&4&\bullet&\bullet&6\\
5&7&\bullet&8&\bullet\\
6&9&10&\bullet&\bullet\\\hline
7&10&9&11&\bullet\\
8&12&7&\bullet&3\\
9&13&14&15&\bullet\\
10&15&8&16&\bullet\\
11&\bullet&11&13&5\\
12&\circ&\bullet&17&18\\
13&19&20&\bullet&21\\
14&21&12&\bullet&19\\
15&22&\bullet&18&17\\
16&16&\bullet&22&1\\
17&14&6&\bullet&20\\
\end{array}$}\\[5pt]
$N=45,\ H=1$.
\end{minipage}\hfill
\begin{minipage}[t]{.24\textwidth}\centering
\textbf{$m=18$}\\[5pt]
{\footnotesize\setlength{\arraycolsep}{2.4pt}\renewcommand{\arraystretch}{1.12}$\begin{array}{r|rrrr}
 &1&2&3&4\\\hline
1&\bullet&\bullet&1&2\\
2&\bullet&3&\bullet&4\\
3&\bullet&5&2&\bullet\\
4&4&\bullet&\bullet&6\\
5&7&\bullet&8&\bullet\\
6&9&10&\bullet&\bullet\\\hline
7&10&9&11&\bullet\\
8&12&\bullet&12&13\\
9&14&7&\bullet&3\\
10&15&16&17&\bullet\\
11&17&8&18&\bullet\\
12&\bullet&11&15&5\\
13&13&\bullet&19&20\\
14&21&22&\bullet&23\\
15&23&14&\bullet&21\\
16&24&\bullet&20&19\\
17&18&\bullet&24&1\\
18&16&6&\bullet&22\\
\end{array}$}\\[5pt]
$N=48,\ H=0$.
\end{minipage}
\caption{Historical witnesses from the finite discovery stage. They all meet the cell bound. Only the sixteen-row witness is needed as a seed in the uniform proof; its full elementary verification is supplied in the appendix.}\label{fig:finite-bases}
\end{figure}

\section{An eventual formula from a sixteen-row double chain}\label{sec:general}

This section proves the main construction theorem. The proof has two distinct tasks. First, we resolve every selected two-edge, meaning that its two occupied cells are certified to carry the same label. Second, we certify orthogonality between every two distinct labels. The construction is organized as follows: a fixed sixteen-row core is extended by two parallel chains; the chain labels are resolved in a prescribed order; path transport proves orthogonality inside the new part; and a column-wise argument proves orthogonality between the new part and the old core. The case $t=0$ is a separate finite base case, verified in Appendix~\ref{app:seed}.

\begin{center}
\begin{minipage}{0.92\linewidth}
\textbf{Proof roadmap.} The ordinary skeleton gives the upper bound. The construction then proceeds through five checkpoints:
\begin{enumerate}[leftmargin=2em]
\item define the fixed core and the two-chain family $K_t$;
\item resolve the retained core pairs and then the new chain pairs;
\item prove orthogonality inside the two chains by grounded path transport;
\item prove the old--new interface, first in column four and then in all columns;
\item replay the seed certificate and obtain the even and odd orders.
\end{enumerate}
A reduction arrow $P\leadsto Q$ is not an equality: it means that one genuine rectangle rule transfers the obligation $P$ to the obligation $Q$. Every reduction chain is read backwards from a terminal relation already grounded by a common row, a common column, or an earlier lemma.
\end{minipage}
\end{center}

\subsection{The ordinary skeleton and the upper bound}
The first step is independent of the recursive certificate. It identifies the ordinary extremal skeleton and converts the cell count into the upper bound used later.
\begin{lemma}[The four-column ordinary skeleton, see {\cite{Culik56}}]\label{lem:skeleton}
For $m\ge6$, $z(m,4)=m+6$. Every extremal one-edge graph consists, up to row and column permutations, of the six degree-two rows with column pairs $12,13,14,23,24,34$, and $m-6$ degree-one rows.
\end{lemma}


\begin{corollary}[Cell bound, see {\cite{LQrevision26}[Corollary 3.4]}]\label{cor:cell}
For $m\ge6$, every configuration admitted in Definition~\ref{def:zrl} satisfies
\begin{equation}\label{eq:cell-new}
N(G)=\frac{5m+6-H}{2},\qquad H\equiv m\pmod2.
\end{equation}
Consequently,
\begin{equation}\label{eq:upper-new}
\zr(m,4)\le z_2(m,4)\le\floor{\frac{5m+6}{2}}.
\end{equation}
\end{corollary}


\subsection{A fixed core, two parallel chains, and a terminal row}
The construction has three visible parts: the old fifteen-row core, two parallel paths of two-edges, and a terminal row. The symbols $A_i$ and $S_i$ denote the two paths, while $R_i$ joins the two rows in the $i$th new block.
We now give one construction for every even $m\ge16$. The odd orders will follow by deleting one fixed row. 

Let $K_0$ be the sixteen-row seed in Figure~\ref{fig:seed}. The integer $j$ denotes the two-edge $B_j$, and every $\bullet$ is a different one-edge. This is the earlier sixteen-row witness after the column permutation $(2,3,1,4)$. The complete elementary verification of the seed is given in Appendix~\ref{app:seed}. In particular, its admissibility does not require running an external checker.

\begin{figure}[ht]\centering
\begin{minipage}[t]{.46\linewidth}\centering\small
\textbf{The normalized seed $K_0$}\par\medskip
\renewcommand{\arraystretch}{1.12}
$\begin{array}{c|cccc}&1&2&3&4\\\hline
1&\bullet&1&\bullet&1\\
2&\boxed{2}&\bullet&\bullet&3\\
3&\boxed{4}&5&\bullet&\bullet\\
4&\bullet&\bullet&6&7\\
5&\bullet&9&8&\bullet\\
6&6&\bullet&10&\bullet\\
\hline
7&12&\bullet&11&13\\
8&7&\bullet&13&11\\
9&15&16&14&\bullet\\
10&17&14&16&\bullet\\
11&18&10&15&\bullet\\
12&19&17&19&\bullet\\
13&3&8&12&\bullet\\
14&21&21&20&\bullet\\
15&5&20&18&\bullet\\
\hline
16&9&\boxed{4}&\boxed{2}&\bullet\\\end{array}$
\par\medskip $m=16,\quad N=43,\quad H=0.$
\end{minipage}\hfill
\begin{minipage}[t]{.49\linewidth}\centering\small
\textbf{The regular twenty-row member $K_2$}\par\medskip
\renewcommand{\arraystretch}{1.12}
$\begin{array}{c|cccc}&1&2&3&4\\\hline
1&\bullet&1&\bullet&1\\
2&A_{0}&\bullet&\bullet&3\\
3&S_{0}&5&\bullet&\bullet\\
4&\bullet&\bullet&6&7\\
5&\bullet&9&8&\bullet\\
6&6&\bullet&10&\bullet\\
\hline
7&12&\bullet&11&13\\
8&7&\bullet&13&11\\
9&15&16&14&\bullet\\
10&17&14&16&\bullet\\
11&18&10&15&\bullet\\
12&19&17&19&\bullet\\
13&3&8&12&\bullet\\
14&21&21&20&\bullet\\
15&5&20&18&\bullet\\
\hline
16&A_{1}&\bullet&A_{0}&R_{1}\\
17&S_{1}&S_{0}&\bullet&R_{1}\\
\hline
18&A_{2}&\bullet&A_{1}&R_{2}\\
19&S_{2}&S_{1}&\bullet&R_{2}\\
\hline
20&9&S_{2}&A_{2}&\bullet\\\end{array}$
\par\medskip $m=20,\quad N=53,\quad H=0.$
\end{minipage}
\caption{One seed and two parallel paths. Integers denote $B_j$ and black dots are distinct one-edges. The boxed pairs $B_2,B_4$ in $K_0$ are lengthened into the $A$- and $S$-paths. The first fifteen rows remain fixed apart from their two port labels. The twenty-row member consists of that fixed core, two identically patterned two-row blocks, and a terminal row; it is not the irregular twenty-row seed of the earlier construction.}\label{fig:seed}
\end{figure}

The first six rows constitute the classical degree-two core; the fixed fifteen-row part also includes nine degree-one rows. The two edges to be lengthened are
\begin{equation}\label{eq:two-ports}
B_2=\{(2,1),(16,3)\},\qquad
B_4=\{(3,1),(16,2)\}.
\end{equation}
For $t\ge1$, retain the first fifteen rows, replacing $(2,1)$ by $A_0$ and $(3,1)$ by $S_0$. Retain every $B_j$ except $B_2,B_4$. Set
\[
m=16+2t,\qquad u_i=14+2i,\quad v_i=15+2i\quad(1\le i\le t),\qquad w=16+2t.
\]
Insert the following rows between the fixed core and the last row:
\begin{equation}\label{eq:double-chain}
\begin{array}{c|cccc}
 &1&2&3&4\\\hline
u_i&A_i&U_i&A_{i-1}&R_i\\
v_i&S_i&S_{i-1}&V_i&R_i\\\hline
w&B_9&S_t&A_t&W
\end{array}
\end{equation}
Here $U_i,V_i,W$ are distinct one-edges, not paired labels. Each $R_i$ is a column-degenerate two-edge. The $A$-edges form a path using columns one and three; the $S$-edges form a second path using columns one and two. In coordinates, put
\[
a_0=2,\quad a_i=u_i\ (1\le i\le t),\quad a_{t+1}=w,
\qquad
s_0=3,\quad s_i=v_i\ (1\le i\le t),\quad s_{t+1}=w.
\]
Then the replacement pairs are exactly
\begin{equation}\label{eq:double-pairs}
\begin{aligned}
A_i&=\{(a_i,1),(a_{i+1},3)\},\quad
S_i=\{(s_i,1),(s_{i+1},2)\} &&(0\le i\le t),\\
R_i&=\{(u_i,4),(v_i,4)\} &&(1\le i\le t).
\end{aligned}
\end{equation}
Thus two old pairs are replaced by $3t+2$ pairs. No old occupied cell becomes a hole. Denote the resulting configuration by $K_t$. Its counts are
\begin{equation}\label{eq:new-counts}
|E_1(K_t)|=22+2t=m+6,\quad
|E_2(K_t)|=21+3t,\quad
N(K_t)=43+5t=\frac{5m+6}{2},\quad H(K_t)=0.
\end{equation}
The new one-edges are $(u_i,2),(v_i,3)$, so they introduce no new one-edge quadrilateral. Simplicity follows immediately from the displayed supports. In the displayed column order, the degree-one rows are distributed as $(0,t+2,t,8)$; this count is not sorted into a canonical skeleton type.

The operation from $K_t$ to $K_{t+1}$ lengthens both terminal pairs before the last row. It adds two rows, two one-edges, and three net two-edges. It is not a claim that an arbitrary two-row block can be appended to an arbitrary admissible graph.

\subsection{Resolving all pairs without a computer-assisted initialization}
The next two lemmas concern resolution only. They do not yet prove all orthogonality obligations. The first lemma initializes the fixed core; the second resolves the double chain in an order forced by its dependency structure.
For the rest of the proof with $t\ge1$, an \emph{old row} means one of the first fifteen rows or the last row. We give the last row the old name $16$ when referring to the seed pattern, but its physical row number is $w$. Let $p_{rc}$ denote the occupied cell in old row $r$, column $c$. At a black dot write $x_{rc}$ for its one-edge label. This convention does not identify any pair prematurely.

A reduction arrow below means that a genuine rectangle transfers the obligation on the left to that on the right. It is read backwards, starting with its grounded terminal pair. Saturation may use a second representative only after the corresponding edge has been resolved. Thus the tables below are finite mathematical derivations, not compressed computer traces.

\begin{lemma}[Elementary core initialization]\label{lem:init-new}
In every $K_t$, $t\ge1$, the sixteen retained two-edges
\[
B_j,\qquad j\in\{1,5,6,7,10,11,12,13,14,15,16,17,18,19,20,21\},
\]
resolve using only the first fifteen rows. No identification of $A_0,A_t,S_0,S_t$, or of $B_3,B_8,B_9$, is used.
\end{lemma}
This lemma supplies the ground layer for the later construction. It deliberately avoids the two opened port pairs and the three pairs that will be reached through the double chain.
\begin{proof}
The line rule resolves $B_1,B_{19},B_{21}$. The complementary-pair rule resolves $B_{11},B_{13}$ on rows $7,8$ and columns $3,4$, and $B_{14},B_{16}$ on rows $9,10$ and columns $2,3$.

The remaining nine identifications follow in the order of Table~\ref{tab:core-init}. For each row, first certify the companion diagonal by reading its reduction chain backwards; its last pair has already resolved representatives on a common line. Then the rectangle of $B_j$ certifies value one and identifies its two halves. For example, the companion reductions for $B_5$ use successively rows $3,14$, columns $1,3$, and rows $3,14$, columns $2,3$; the terminal $B_{20}$ has its other representative at $(15,2)$. The cell $p_{32}$ need not yet be identified with $p_{15,1}$. Thus there is no circular use of the identification being proved. All other rows have exactly the same explicitly checkable interpretation.

\begin{table}[ht]
\centering\small
\setlength{\tabcolsep}{5pt}\renewcommand{\arraystretch}{1.3}
\begin{tabular}{c l}\hline
Resolve&Companion reductions; the final pair is grounded by a line\\\hline
$B_{17}$&$(B_{14},B_{19})$\\[3pt]
$B_{15}$&$(B_{14},p_{11,1})\leadsto (B_{17},p_{11,2})$\\[3pt]
$B_{18}$&$(B_{15},p_{15,1})$\\[3pt]
$B_{20}$&$(B_{18},B_{21})$\\[3pt]
$B_{5}$&$(B_{20},p_{3,1})\leadsto (B_{21},x_{3,3})\leadsto (B_{20},p_{3,2})$\\[3pt]
$B_{10}$&$(B_{15},x_{6,2})\leadsto (B_{16},p_{6,1})\leadsto (B_{17},p_{6,3})\leadsto (B_{19},x_{6,2})\leadsto (B_{17},p_{6,1})$\\[3pt]
$B_{6}$&$(B_{10},x_{4,1})\leadsto (B_{18},x_{4,2})\leadsto (B_{20},p_{4,3})$\\[3pt]
$B_{7}$&$(B_{11},x_{4,1})\leadsto (B_{6},p_{7,1})$\\[3pt]
$B_{12}$&$(B_{11},p_{13,1})\leadsto (B_{7},x_{13,4})$\\[3pt]
\hline\end{tabular}
\caption{A nine-row mathematical initialization for the fixed core. Raw cell symbols $p_{rc}$ are not identified in advance. A single entry in the second column means that the companion is already grounded.}\label{tab:core-init}
\end{table}

Each use of $p_{rc}$ refers to the indicated cell, even if its two-edge has not yet resolved. Every $B_j$ used as a movable label has resolved on an earlier line. The table therefore gives a finite mathematical derivation, rather than a reference to a stored program trace.
\end{proof}

\begin{lemma}[Uniform resolution of the double chain]\label{lem:resolution-new}
Every two-edge of $K_t$ resolves for every $t\ge1$.
\end{lemma}
The left-to-right order for the $S$-path and the right-to-left order for the $A$-path are forced: each step uses a pair resolved at the preceding step. The endpoint cases are stated separately because the next block does not exist there.
\begin{proof}
All $R_i$ resolve by the column rule. Resolve $S_0,S_1,\ldots,S_t$ from left to right. For $S_0$, the companion diagonal is $B_5$ at $(3,2)$ and $(v_1,1)$; $B_5$ also has a representative in column one. For $1\le i<t$, the companion for $S_i$ contains $S_{i-1}$ and $(v_{i+1},1)$, and $S_{i-1}$ already has a column-one representative. At the endpoint the second cell is $B_9$ at $(w,1)$, with the same conclusion. Hence every $S_i$ resolves.

Next $B_9=\{(5,2),(w,1)\}$ resolves: its companion is $x_{51}$ and $S_t$ at $(w,2)$, and $S_t$ has a column-one representative. Then $B_8=\{(5,3),(13,2)\}$ resolves, since its companion labels $B_9,B_{12}$ both have column-one representatives.

Resolve $A_t$ first. Its companion is $(u_t,3)$ and $B_9$ at $(w,1)$. Instead use $B_9$ at $(5,2)$: the corresponding companion pair is $U_t$ and $B_8$ at $(5,3)$. These are orthogonal because $B_8$ also occurs at $(13,2)$. Thus $A_t$ resolves. For $i=t-1,\ldots,0$, the companion for $A_i$ contains $A_{i+1}$ and a column-three cell in the preceding path row. The resolved $A_{i+1}$ has a column-three representative, so this companion is grounded. This includes $i=0$, where the preceding path row is old row $2$.

Only $B_3=\{(2,4),(13,1)\}$ remains. Its companion labels are $A_0,x_{13,4}$. They have the reduction chain
\[
(A_0,x_{13,4})\ \leadsto\ (R_1,B_{12})\ \leadsto\ (A_1,B_{13}).
\]
The first rectangle uses $A_0$ at $(u_1,3)$ and $x_{13,4}$; the second uses $R_1$ at $(u_1,4)$ and $B_{12}$ at $(7,1)$. The terminal pair shares column three, since $B_{13}$ occurs at $(8,3)$. Therefore $B_3$ resolves. Every identification used above joins the two halves of a prescribed pair, so distinct selected edges stay distinct classes.
\end{proof}

\subsection{A path-transport lemma and orthogonality inside the chain}
We now change tasks: all relevant pairs have been resolved, and we begin proving orthogonality. The defect notation below treats ordinary orthogonality and the unit-norm identity in one formula.
From now on a resolved label also denotes its common coefficient vector. To avoid a possible confusion when a companion is a selected diagonal, set
\[
[X,Y]=\langle X,Y\rangle-\mathbf 1_{X=Y}.
\]
For distinct labels, $[X,Y]=0$ is ordinary orthogonality; for the same label it is the already known unit-norm identity. A rectangle gives equality of one defect with the negative of the companion defect. All uses of such identities below terminate at a known correct value. We do not infer a new certificate merely from an ungrounded equation of the form $f=-f$.

\begin{lemma}[Grounded path transport]\label{lem:path}
Let $r_0,\ldots,r_k$ be distinct rows and $a,b,c$ distinct columns. Suppose the pairs
\[
e_j=\{(r_j,a),(r_{j+1},b)\},\qquad 0\le j<k,
\]
are resolved, and that all side cells $y_i=(r_i,c)$ are occupied. For every $0\le i\le k$ and $0\le j<k$, the closure certifies $y_i\perp_R e_j$.
\end{lemma}
The point of this lemma is distance reduction. A far-away inner product is transferred to one with a smaller index gap, and the process ends at a same-row zero. The word ``grounded'' means that the transfer chain terminates at such a known zero rather than at an ungrounded relation of the form $f=-f$.
\begin{proof}
Write $f_{ij}=\langle y_i,e_j\rangle$. The target labels are distinct, since no $e_j$ uses column $c$. Same-row relations ground
\[
f_{j,j}=f_{j+1,j}=0.
\]
If $i<j$, the rectangle on rows $r_i,r_j$ and columns $c,a$ gives $f_{ij}=-f_{ji}$. If $i=j-1$, the latter is grounded. If $i<j-1$, a second rectangle on rows $r_{i+1},r_j$ and columns $b,c$ gives
\begin{equation}\label{eq:path-shortening}
f_{ij}=f_{i+1,j-1}.
\end{equation}
The index gap decreases by two. If $i>j+1$, the rectangle on rows $r_{j+1},r_i$ and columns $b,c$ gives $f_{ij}=-f_{j+1,i-1}$, reducing to a grounded or previously treated case. Thus every target has a finite transfer path to a same-row zero. Reading that path backwards is a valid closure derivation.
\end{proof}

Apply Lemma~\ref{lem:path} to the $A$-path with side columns $2,4$, and to the $S$-path with side columns $3,4$. Besides all path-side relations, this gives
\begin{equation}\label{eq:path-boundaries}
A_i\perp_R x_{22},B_3,W,S_t,
\qquad
S_i\perp_R x_{33},x_{34},W,A_t
\quad(0\le i\le t).
\end{equation}
Both types of path edge are orthogonal to every $R_j$. Pairs of $A$-edges, pairs of $S$-edges, and a pair $A_i,S_j$ share column one. For the remaining types of new edges, the companion pairs are
\[
\begin{array}{c|c}
\text{target}&\text{companion labels}\\\hline
U_i,V_j&A_{i-1},S_{j-1}\\
U_i,R_j&R_i,S_{j-1}\\
V_i,R_j&R_i,A_{j-1}
\end{array}
\]
using $R_j$ in row $v_j$ for the middle line and in row $u_j$ for the last. Edges of the same type $U,V,R$ share a column. Consequently all distinct labels in
\begin{equation}\label{eq:chain-set}
\mathcal C=\{A_0,\ldots,A_t,S_0,\ldots,S_t\}
\cup\{U_i,V_i,R_i:1\le i\le t\}
\end{equation}
are certified mutually orthogonal.

\subsection{The interface with the old grid: a column-wise proof}
The new labels are already mutually orthogonal. We next connect them to the old grid. We first prove all relations with the old fourth column, because the other columns can then be reduced to that column by rectangles.
The following argument replaces the former large symbolic interface certificate. Its induction has only sixteen old-row cases, followed by a short boundary calculation and column completion. Write
\[
X_{rc}=\text{the label in old cell }(r,c),\qquad F_r=X_{r4},\qquad F_{16}=W.
\]
Statements involving an old label and a chain label are always understood for distinct labels. If both names denote the same edge, its prescribed value one is already certified.

\begin{lemma}[Fourth-column propagation]\label{lem:fourth}
For every $t\ge1$, every $0\le i\le t$, and every old row $r$, the closure certifies
\begin{equation}\label{eq:fourth-zero}
[A_i,F_r]=[S_i,F_r]=0.
\end{equation}
\end{lemma}
Here $T_i(r)$ is a bookkeeping device: the same companion pair expresses both $[A_i,F_r]$ and $[S_i,F_r]$. The induction propagates the zero value of $T_i(r)$ through the sixteen old rows and through the chain index $i$.
\begin{proof}
\emph{Step 1: three elementary transfer laws.}
For $1\le i\le t$, the two rectangles using old row $r$ and rows $u_i,v_i$ give
\begin{equation}\label{eq:bridge}
[A_i,F_r]=-[R_i,X_{r1}]=[S_i,F_r]=:T_i(r).
\end{equation}
For a retained old pair with the indicated supports, the same argument yields
\begin{align}
B_j=\{(r,1),(s,3)\}&\quad\Longrightarrow\quad
T_i(r)=[A_{i-1},F_s],\label{eq:old13}\\
B_j=\{(r,1),(s,2)\}&\quad\Longrightarrow\quad
T_i(r)=[S_{i-1},F_s].\label{eq:old12}
\end{align}
A pair with support $\{(r,3),(s,2)\}$ gives
\begin{equation}\label{eq:old23}
T_i(r)=T_i(s).
\end{equation}
Indeed, for $i<t$ both sides reduce to $-[R_{i+1},B_j]$, using the column-three representative of $A_i$ and the column-two representative of $S_i$. For $i=t$ use their terminal representatives in row $w$; both sides then reduce to $-[W,B_j]$. Thus (\ref{eq:old23}) remains valid at the last block, without assuming a nonexistent next block. These are pairs of ordinary rectangle transfers, not solutions of a simultaneous linear system.

For $i\ge1$, $A_i$ and $S_i$ are orthogonal to every retained $B_j$. Common columns prove this except for $A_i,B_1$ and $S_i,B_{11},B_{13}$. Each exceptional $B_j$ occurs in column four, and (\ref{eq:bridge}) changes to the other chain type, which does have a common column. This settles the exceptions.

\emph{Step 2: the $A_0$ boundary.}
Use $A_0$ at $(2,1)$ and $F_r$ at $(r,4)$. The companion is $B_3$ and $X_{r1}$, whose prescribed value is known from column one, because $B_3$ also occurs at $(13,1)$. When $r=13$, this companion is the selected pair $B_3$ and has value one, not zero; the defect formulation gives the same valid transfer. For $r=2$ the target is already a same-row pair. Hence
\begin{equation}\label{eq:A0-fourth}
[A_0,F_r]=0\quad(1\le r\le16).
\end{equation}

\emph{Step 3: the sixteen old rows.}
For $i\ge1$ the following table proves $T_i(r)=0$. Process $i=1,2,\ldots,t$, and within each $i$ read the rows in the displayed order. A reference to $A_{i-1}$ is grounded by (\ref{eq:A0-fourth}) when $i=1$, and otherwise by the preceding induction step.
\begin{center}\small
\renewcommand{\arraystretch}{1.18}
\begin{tabular}{c p{10.6cm}}\hline
Old row $r$&Grounding or reduction\\\hline
$1,2,4,7,8$&$F_r$ is respectively $B_1,B_3,B_7,B_{13},B_{11}$; use Step 1.\\
$3,16$&The $S$-path gives $S_i\perp_R x_{34},W$; use (\ref{eq:bridge}).\\
$13$&$T_i(13)=-[R_i,B_3]=0$, since both labels have column-four representatives.\\
$5$&$T_i(5)=T_i(13)$ by $B_8$ and (\ref{eq:old23}).\\
$6$&$T_i(6)=[A_{i-1},F_4]=[A_{i-1},B_7]=0$ by $B_6$ and a common column.\\
$11$&$T_i(11)=T_i(6)$ by $B_{10}$ and (\ref{eq:old23}).\\
$9$&$T_i(9)=[A_{i-1},F_{11}]=0$ by $B_{15}$ and induction.\\
$10$&$T_i(10)=T_i(9)$ by $B_{14}$ and (\ref{eq:old23}).\\
$12$&$T_i(12)=[A_{i-1},F_{12}]=0$ by $B_{19}$ and induction.\\
$15$&$T_i(15)=[S_{i-1},F_3]=0$ by $B_5$ and the $S$-path.\\
$14$&$T_i(14)=T_i(15)$ by $B_{20}$ and (\ref{eq:old23}).\\\hline
\end{tabular}
\end{center}
For instance, the within-row pair $B_{19}$ is precisely what gives the simple recurrence at row $12$. A row-degenerate pair is therefore an aid to this proof, not an irregularity to be excluded.

\emph{Step 4: the $S_0$ boundary.}
Three short calculations handle the exceptional old column-two singletons. First, all pairs in the chain
\begin{equation}\label{eq:boundary-four}
(A_0,x_{82})\leadsto(x_{22},B_7)\leadsto(B_3,x_{42})
\leadsto(B_8,x_{41})\leadsto(B_6,x_{51})
\end{equation}
are certified by reading backwards from the terminal common-column pair. The rectangles use, respectively, rows $2,8$ and columns $1,2$; rows $2,4$ and columns $2,4$; rows $4,13$ and columns $1,2$; and rows $4,5$ and columns $1,3$.

Next $S_0\perp_R B_{11}$ follows from rows $3,7$, columns $1,3$, whose companion is $x_{33},B_{12}$. The latter share column three. The reductions
\begin{align}
(S_0,B_{13})&\leadsto(R_1,x_{72})\leadsto(U_1,B_{13})
\leadsto(A_0,x_{82}),\label{eq:S0B13}\\
(S_0,F_6)&\leadsto(x_{34},B_6)\leadsto(x_{33},B_7)
\leadsto(S_0,B_{13})\label{eq:S0F6}
\end{align}
are now grounded. In (\ref{eq:S0B13}), the consecutive rectangles use rows $v_1,7$, columns $2,4$; rows $u_1,7$, columns $2,4$; and rows $u_1,8$, columns $2,3$. Equation (\ref{eq:S0F6}) uses rows $3,6$, columns $1,4$; rows $3,4$, columns $3,4$; and rows $3,8$, columns $1,3$.

At this stage $R_1$ is orthogonal to every old cell in columns $1,3,4$: rectangle transfer reduces these targets to $[A_1,F_r]$, $[A_0,F_r]$, or a same-column relation. In old column two, every nonsingleton label has another representative in one of these three columns. Thus $[R_1,X_{r2}]=0$ unless $r\in\{2,4,6,7,8\}$. Transferring via row $v_1$ gives $[S_0,F_r]=0$ for all other rows. For the five exceptions, $F_2=B_3$ and $F_4=B_7$ share column one with $S_0$; $F_6$ is handled by (\ref{eq:S0F6}); and $F_7=B_{13},F_8=B_{11}$ were handled above. This completes every boundary case.
\end{proof}

\begin{lemma}[Completion of the interface]\label{lem:interface-new}
Every chain label in $\mathcal C$ is orthogonal to every distinct label occurring in the old rows.
\end{lemma}
This final interface lemma is best read in four blocks: first $R_i$, then $A_i$, then $S_i$, and finally $U_i,V_i$. In each block, common-column cases are immediate; only the exceptional singleton cells require a displayed reduction chain.
\begin{proof}
First take $R_i$ and an old cell $X_{rc}$. For $c=1,2,3$, use $R_i$ in the appropriate row of its block. Its companion pair is, respectively,
\[
(A_i,F_r),\qquad(S_{i-1},F_r),\qquad(A_{i-1},F_r),
\]
all certified by Lemma~\ref{lem:fourth}. Column four is immediate. Thus every $R_i$ is orthogonal to every old label.

For $A_i$, common columns handle old columns one and three, and Lemma~\ref{lem:fourth} handles column four. Every retained two-edge has a representative in these columns; the port labels are already handled by (\ref{eq:chain-set}). It remains to treat only the five old one-edges in column two, at rows $2,4,6,7,8$. The first is covered by the $A$-path lemma. The others have the following reductions; statements at $i=t$ are given separately so that every endpoint is explicit.
\begin{center}\small
\renewcommand{\arraystretch}{1.2}
\begin{tabular}{c p{10.6cm}}\hline
Target&Reduction and grounding\\\hline
$A_i,x_{62}$&For $i<t$ the companion is $U_{i+1},B_{10}$, a common-column pair; for $i=t$ it is $S_t,B_{10}$, also sharing column two.\\
$A_i,x_{42}$&For $i<t$ reduce to $U_{i+1},B_6$, then to $A_{i+1},x_{62}$; for $i=t$ the companion $S_t,B_6$ shares column one.\\
$A_i,x_{72}$&Using the column-one representative, reduce first to $Y_i,B_{12}$, then to $Z_i,B_8$, where $(Y_i,Z_i)=(U_i,A_{i-1})$ for $i\ge1$ and $(Y_0,Z_0)=(x_{22},x_{23})$. The last pair shares column three.\\
$A_i,x_{82}$&The case $i=0$ is (\ref{eq:boundary-four}). For $1\le i<t$, reduce to $U_{i+1},B_{13}$ and then to $R_{i+1},x_{72}$. For $i=t$, the companion is $S_t,B_{13}$, handled in Step 1 of Lemma~\ref{lem:fourth}.\\\hline
\end{tabular}
\end{center}
These reductions are valid also at $t=1$; no interior block is required.

For $S_i$, only the old column-three one-edges at rows $1,2,3$ remain. Row $3$ is covered by the $S$-path. For row $2$, use the column-one representative of $S_i$: its companion is $V_i,A_0$ when $i\ge1$, and $x_{33},A_0$ when $i=0$, both common-column pairs. For row $1$ and $i<t$, using the column-two representative gives $V_{i+1},B_1$; transfer using $B_1$ at $(1,4)$ reduces to $R_{i+1},x_{13}$. If $i=t$, the companion is $A_t,B_1$, already settled. The retained pairs $B_{11},B_{13}$ at $i=0$ were explicitly handled in Lemma~\ref{lem:fourth}; hence there is no omitted disjoint-column case.

Finally, a target $U_i,X_{rc}$ with $c\ne2$ transfers to $X_{r2}$ paired with $A_i,A_{i-1}$, or $R_i$, according as $c=1,3,4$. A target $V_i,X_{rc}$ with $c\ne3$ transfers to $X_{r3}$ paired with $S_i,S_{i-1}$, or $R_i$. All of these relations have just been proved. Same-column cases are immediate. This completes the interface by four column cases, rather than by a large precomputed list of old-label obligations.
\end{proof}

\subsection{Completing the proof and obtaining every odd order}
The replay lemma reconnects the extended construction to the finite seed certificate. Once the even-order family is certified, deleting one fixed old row gives the odd-order family and leaves exactly one hole.
\begin{lemma}[Replaying an old certificate after splitting pairs]\label{lem:split-replay}
Suppose an admissible old grid is embedded with the same occupied cells, but some selected pairs are split into distinct resolved port labels. Assume that all retained old pairs resolve, and that every port label has its correct relation to every old label, including the other ports. Then all relations between distinct retained old labels follow by replaying the old certificate.
\end{lemma}
\begin{proof}
Proceed by induction on the old derivation. Omit the identifications of pairs that were split. A conclusion involving a port is already supplied by the hypothesis. Identifications and saturation involving retained labels are unchanged. For a rectangle concluding a relation between retained labels, a companion not involving a port retains its old prescribed value and is supplied by induction. A companion involving ports has its \emph{new} prescribed value by hypothesis. In particular, a formerly selected diagonal may now have prescribed value zero rather than one; it is incorrect to keep its old value, but the rectangle rule requires only the new correct value, which is available. Thus every required target relation between retained labels survives. This is a derivation-level argument and does not require enumerating an old finite trace again.
\end{proof}

\begin{proposition}[All even orders from sixteen]\label{prop:even-new}
Every $K_t$, $t\ge0$, satisfies $\rw$. Consequently, $\zr(16+2t,4)=43+5t$.
\end{proposition}
\begin{proof}
The elementary seed proof for $t=0$ is in Appendix~\ref{app:seed}. For $t\ge1$, Lemmas~\ref{lem:init-new} and \ref{lem:resolution-new} resolve every pair without identifying distinct edges. The path argument settles every pair of distinct chain labels, and Lemma~\ref{lem:interface-new} settles all old--new obligations.

Apply Lemma~\ref{lem:split-replay} to the seed, with $B_2$ replaced by the two port labels $A_0,A_t$ and $B_4$ replaced by $S_0,S_t$. For $t\ge1$ these four labels are distinct. All port relations required by the lemma have been proved, including the relations between the two halves of each opened seed pair. Hence all old--old obligations are certified as well. Counts (\ref{eq:new-counts}) and the upper bound complete the extremal assertion.
\end{proof}

\begin{proposition}[All odd orders from fifteen]\label{prop:odd-new}
For every $t\ge0$, deleting old row $12$ of $K_t$, together with the complete two-edges $B_{19},B_{17}$, gives an admissible extremal graph on $15+2t$ rows with one hole and $40+5t$ selected edges.
\end{proposition}
\begin{proof}
Old row $12$ is $(B_{19},B_{17},B_{19},x_{12,4})$. Delete $B_{19}$ and $B_{17}$ completely, and delete the one-edge $x_{12,4}$. The row becomes empty and may be removed by Lemma~\ref{lem:delete}. The other half of $B_{17}$, at old $(10,1)$, is the unique hole; both halves of $B_{19}$ lie in the deleted row. Thus the counts are
\[
m=15+2t,\quad |E_1|=21+2t=m+6,\quad |E_2|=19+3t,
\quad H=1,\quad N=40+5t=\floor{\frac{5m+6}{2}}.
\]
The one-edge graph remains quadrilateral-free, so the limited condition and all admissibility requirements hold. The upper bound proves optimality.
\end{proof}

\begin{proof}[Proof of Theorem~\ref{thm:main}]
Propositions~\ref{prop:even-new} and \ref{prop:odd-new} cover all integers $m\ge15$, with no exceptional base orders. They attain (\ref{eq:upper-new}), hence both $\zr$ and $z_2$ equal the stated integer. Equation (\ref{eq:cell-new}) then forces every extremal graph, not just the displayed family, to have $H=0$ at even orders and $H=1$ at odd orders. No classification up to isomorphism is claimed.
\end{proof}

\section{Conclusions and further directions}\label{sec:conclusion}
Under the strengthened definition, the four-column recursive-line problem has an exact formula for all $m\ge15$. The one-edge skeleton is classical, but the augmented problem has a global condition: resolving every two-edge is not enough; distinct edge classes must also be certified orthogonal. The finite computations show this distinction. The double-chain construction turns the zero-or-one-hole pattern into an infinite family. The same formula determines $z_2$ in this range, but it does not settle the full biquadratic SOS-rank problem.

\subsection{The exceptional order and the stabilization threshold}
The most immediate finite question is the following.
\begin{problem}\label{prob:14}
Is $\zr(14,4)=37$ or $\zr(14,4)=38$? Equivalently, does any extremal $14\times4$ one-edge skeleton admit a hole-free pairing of its thirty-six free cells satisfying $\rw$?
\end{problem}
A thirty-eight-edge witness would settle the question positively. A negative answer requires a proof excluding every one of the fifteen skeleton types, not just an unsuccessful heuristic search or the completed exclusion of $(0,0,0,8)$.

There is a sharper consequence of the existing table. Define the stabilization threshold
\[
M_* = \min\{M\ge6:\ \zr(m,4)=\floor{(5m+6)/2}\text{ for every }m\ge M\}.
\]
Theorem~\ref{thm:main} ensures existence. Since the bound is not attained at $m=12$, is attained at $m=13$, and is attained for all $m\ge15$, the present results imply
\[
M_*\in\{13,15\}.
\]
More precisely, $M_*=13$ if $\zr(14,4)=38$, and $M_*=15$ if $\zr(14,4)=37$. Thus settling a single finite case determines the least eventual threshold; the theorem alone does not prove that fifteen is minimal.

\subsection{Structure, certificates, and higher widths}
A structural classification of the extremal augmented graphs remains open. The number of holes is forced by Theorem~\ref{thm:main}, but their possible positions, the allowable singleton-row distributions, and the nonisomorphic two-edge patterns are not classified. It is natural to ask whether every sufficiently large extremal graph has a bounded-size portion that can be reduced by a local rewiring operation, or whether genuinely different infinite families occur.

The double-chain construction separates a reusable path-transport mechanism from the finite seed. It remains useful to seek a conceptual classification of seeds that support the same extension, or a shorter structural proof of the sixteen-row seed instead of an explicit finite derivation. Another direction is formal verification of the ordinary mathematical argument. These are simplification and verification questions, not unproved steps in the present infinite-family proof.

For fixed $n\ge5$, the corresponding ordinary skeleton has $\binom n2$ pair-rows once $m\ge\binom n2$, and the cell bound \cite{LQrevision26} becomes
\[
\zr(m,n)\le\floor{\frac{(n+1)m+\binom n2}{2}}.
\]
It is worth determining which widths admit eventual attainment and which admit reusable bounded-interface constructions. Such attainment must not be presumed for every width: the revised three-column theorem gives $\zr(m,3)=2m$, whereas cell counting allows $2m+1$ \cite{LQrevision26}. The four-column result therefore reflects more than a general parity principle.

Finally, the relationship between $z_2$ and $\zr$ remains unresolved outside the families in which both meet a common upper bound. Failure of the recursive certificate does not by itself prove that the displayed SOS representation is reducible. Conversely, the equality proved here does not determine $\bsr(m,4)$, since general SOS forms need not have the doubly simple limited structure. Exact Gram-matrix methods and additional sound coefficient rules may be useful for separating these questions.

\bigskip

\appendix
\section{An elementary verification of the sixteen-row seed}\label{app:seed}
The notation $x_{rc}$ denotes the one-edge at $(r,c)$, with $x_{16,4}=W$. This appendix supplies the finite base used in Section~\ref{sec:general}. The tables are explicit applications of the line and rectangle rules to the displayed sixteen-row matrix; they are not a reference to an external computational decision. They can be checked directly from Figure~\ref{fig:seed}. The optional verification code merely cross-checks the same printed derivations.

\subsection{Identification of all twenty-one pairs}
The core initialization in Lemma~\ref{lem:init-new} is valid in the unsplit seed as well: the table uses only the first fifteen rows, where no prescribed pair value changes. After those sixteen pairs resolve, the remaining five resolve in the following order. As in Table~\ref{tab:core-init}, the second column begins with the companion diagonal of the selected pair and ends with a grounded same-line pair.
\begin{center}\small\renewcommand{\arraystretch}{1.25}
\begin{tabular}{c l}\hline
Resolve&Companion reductions\\\hline

$B_{4}$&$(B_{5},p_{16,1})$\\[3pt]
$B_{9}$&$(B_{4},x_{5,1})$\\[3pt]
$B_{8}$&$(B_{12},B_{9})$\\[3pt]
$B_{2}$&$(B_{9},x_{2,3})\leadsto (B_{8},x_{2,2})$\\[3pt]
$B_{3}$&$(B_{2},x_{13,4})\leadsto (B_{12},x_{16,4})\leadsto (B_{13},B_{9})\leadsto (B_{2},B_{7})$\\[3pt]
\hline\end{tabular}\end{center}
For $B_3$, the last reduction uses rows $8,16$ and columns $1,3$; its companion is $B_7,B_2$, a common-column pair. Thus all twenty-one two-edges resolve. Each identification joins only the two cells of a selected edge.

\subsection{Orthogonality of two-edges}
All but twenty-five of the $\binom{21}{2}=210$ distinct two-edge pairs share a row or column. Table~\ref{tab:seed-BB} handles every exception. For an operator $\rho_{r,s}^{c,d}$, use the genuine rectangle on rows $r,s$ and columns $c,d$, changing the current diagonal to its companion. Operators are applied from left to right, but their certificates are read from the final grounded pair backwards. The initial labels, rectangle word, and final pair therefore specify every step, including the longer reductions. Every final pair is grounded by a common line after the resolutions above.

\begin{table}[t]\centering\small
\setlength{\tabcolsep}{6pt}\renewcommand{\arraystretch}{1.22}
\begin{tabular}{c p{7.7cm} l}\hline
Target&Rectangle word&Grounded pair\\\hline

$(B_{1},B_{2})$&$\rho_{1,2}^{1,4}$&$(B_{3},x_{1,1})$\\[4pt]
$(B_{1},B_{6})$&$\rho_{1,4}^{3,4}\;\rho_{1,8}^{1,3}\;\rho_{1,7}^{1,4}\;\rho_{1,13}^{2,3}$&$(B_{8},x_{1,3})$\\[4pt]
$(B_{1},B_{12})$&$\rho_{1,13}^{2,3}$&$(B_{8},x_{1,3})$\\[4pt]
$(B_{1},B_{15})$&$\rho_{1,11}^{2,3}$&$(B_{10},x_{1,3})$\\[4pt]
$(B_{1},B_{18})$&$\rho_{1,15}^{2,3}$&$(B_{20},x_{1,3})$\\[4pt]
$(B_{1},B_{19})$&$\rho_{1,12}^{1,2}$&$(B_{17},x_{1,1})$\\[4pt]
$(B_{10},B_{3})$&$\rho_{6,13}^{1,3}$&$(B_{12},B_{6})$\\[4pt]
$(B_{14},B_{3})$&$\rho_{9,13}^{1,3}$&$(B_{12},B_{15})$\\[4pt]
$(B_{16},B_{3})$&$\rho_{10,13}^{1,3}$&$(B_{12},B_{17})$\\[4pt]
$(B_{20},B_{3})$&$\rho_{13,14}^{1,3}$&$(B_{12},B_{21})$\\[4pt]
$(B_{11},B_{4})$&$\rho_{3,7}^{1,3}$&$(B_{12},x_{3,3})$\\[4pt]
$(B_{13},B_{4})$&$\rho_{8,16}^{2,3}\;\rho_{2,8}^{1,2}\;\rho_{2,4}^{2,4}\;\rho_{4,13}^{1,2}\;\rho_{4,5}^{1,3}$&$(B_{6},x_{5,1})$\\[4pt]
$(B_{11},B_{5})$&$\rho_{7,15}^{1,3}$&$(B_{12},B_{18})$\\[4pt]
$(B_{13},B_{5})$&$\rho_{8,15}^{1,3}$&$(B_{18},B_{7})$\\[4pt]
$(B_{7},B_{8})$&$\rho_{5,8}^{1,3}\;\rho_{5,7}^{1,4}\;\rho_{5,13}^{3,4}$&$(B_{8},x_{13,4})$\\[4pt]
$(B_{10},B_{7})$&$\rho_{4,6}^{3,4}$&$(B_{6},x_{6,4})$\\[4pt]
$(B_{14},B_{7})$&$\rho_{8,10}^{1,2}$&$(B_{17},x_{8,2})$\\[4pt]
$(B_{16},B_{7})$&$\rho_{8,9}^{1,2}\;\rho_{8,11}^{2,3}$&$(B_{10},B_{13})$\\[4pt]
$(B_{20},B_{7})$&$\rho_{8,15}^{1,2}$&$(B_{5},x_{8,2})$\\[4pt]
$(B_{11},B_{9})$&$\rho_{7,16}^{1,3}$&$(B_{12},B_{2})$\\[4pt]
$(B_{13},B_{9})$&$\rho_{8,16}^{1,3}$&$(B_{2},B_{7})$\\[4pt]
$(B_{11},B_{17})$&$\rho_{7,10}^{1,3}$&$(B_{12},B_{16})$\\[4pt]
$(B_{11},B_{21})$&$\rho_{7,14}^{2,3}$&$(B_{20},x_{7,2})$\\[4pt]
$(B_{13},B_{17})$&$\rho_{8,12}^{2,3}\;\rho_{8,12}^{1,2}$&$(B_{17},B_{7})$\\[4pt]
$(B_{13},B_{21})$&$\rho_{8,14}^{2,3}$&$(B_{20},x_{8,2})$\\[4pt]
\hline\end{tabular}
\caption{All two-edge pairs not already grounded by a common line.}\label{tab:seed-BB}
\end{table}

\subsection{Orthogonality between one-edges and two-edges}
There are $22\cdot21=462$ targets of this type. Of these, 232 share a row or column. Table~\ref{tab:seed-immediate} specifies another 120 one-rectangle reductions to a same-line pair or a pair of two-edges, already settled above. For a target $x_{rc},B_j$, the entry $j^{\epsilon}$ means: use the $\epsilon$-th cell of $B_j$ in increasing lexicographic order of its coordinates, form its rectangle with $(r,c)$, and take the companion diagonal. Thus the superscript is a choice of representative, not a power.

\begin{table}[t]\centering\small
\setlength{\tabcolsep}{8pt}\renewcommand{\arraystretch}{1.23}
\begin{tabular}{c p{11.8cm}}\hline
One-edge&Two-edge and representative choices for one transfer\\\hline

$x_{1,1}$&$8^{2},\;10^{1},\;11^{1},\;13^{1},\;14^{2},\;16^{1},\;20^{2}$\\[4pt]
$x_{1,3}$&$3^{2},\;4^{2},\;5^{2},\;7^{1},\;9^{2},\;17^{2},\;21^{2}$\\[4pt]
$x_{2,2}$&$12^{2},\;15^{2},\;18^{2},\;19^{1}$\\[4pt]
$x_{2,3}$&$4^{2},\;5^{2},\;7^{1},\;9^{2},\;17^{1},\;21^{2}$\\[4pt]
$x_{3,3}$&$3^{2},\;7^{2},\;9^{2},\;17^{2},\;21^{2}$\\[4pt]
$x_{3,4}$&$2^{1},\;9^{2},\;12^{1},\;20^{2}$\\[4pt]
$x_{4,1}$&$1^{2},\;8^{1},\;10^{1},\;11^{2},\;13^{2},\;14^{1},\;16^{2},\;20^{1}$\\[4pt]
$x_{4,2}$&$2^{2},\;11^{2},\;12^{2},\;15^{2},\;18^{2},\;19^{1}$\\[4pt]
$x_{5,1}$&$1^{1},\;10^{2},\;11^{1},\;13^{2},\;14^{2},\;16^{1},\;20^{2}$\\[4pt]
$x_{5,4}$&$2^{1},\;4^{2},\;6^{1},\;12^{2}$\\[4pt]
$x_{6,2}$&$2^{2},\;3^{2},\;11^{1},\;12^{2},\;13^{2},\;15^{2},\;18^{1},\;19^{1}$\\[4pt]
$x_{6,4}$&$2^{1},\;12^{1},\;15^{2}$\\[4pt]
$x_{7,2}$&$2^{2},\;3^{2},\;15^{2},\;18^{1},\;19^{1}$\\[4pt]
$x_{8,2}$&$2^{2},\;3^{2},\;12^{2},\;15^{2},\;18^{2},\;19^{1}$\\[4pt]
$x_{9,4}$&$2^{1},\;6^{1},\;12^{1},\;18^{1}$\\[4pt]
$x_{10,4}$&$2^{1},\;6^{1},\;12^{1},\;19^{1}$\\[4pt]
$x_{11,4}$&$2^{1},\;5^{2},\;6^{1},\;12^{1},\;14^{1}$\\[4pt]
$x_{12,4}$&$2^{1},\;6^{1},\;12^{1},\;14^{2}$\\[4pt]
$x_{13,4}$&$2^{1},\;4^{1},\;5^{2},\;6^{1},\;9^{1},\;15^{1},\;17^{1},\;18^{1},\;19^{1},\;21^{1}$\\[4pt]
$x_{14,4}$&$2^{1},\;6^{1},\;12^{1},\;18^{2}$\\[4pt]
$x_{15,4}$&$2^{1},\;4^{1},\;6^{1},\;12^{1},\;15^{2},\;21^{2}$\\[4pt]
$x_{16,4}$&$5^{1},\;6^{1},\;12^{1}$\\[4pt]
\hline\end{tabular}
\caption{The 120 nontrivial one-transfer targets after all two-edge pairs have been settled. Same-line targets are not repeated.}\label{tab:seed-immediate}
\end{table}

The remaining 110 targets occur in the transfer chains of Tables~\ref{tab:seed-chains-1}--\ref{tab:seed-chains-2}. For a fixed side column $c$, abbreviate $(x_{rc},B_j)$ by $j_r$. An arrow $j_r\leadsto k_s$ means that, for some column $a\ne c$, the displayed grid has $B_j$ at $(s,a)$ and $B_k$ at $(r,a)$; the rectangle on rows $r,s$ and columns $a,c$ gives the reduction. The two subscripts identify the rows, so the arrow is directly checkable in the seed matrix. The last state of each chain is either a same-line target or a target in Table~\ref{tab:seed-immediate}. Every intermediate target is thereby grounded, not just the initial state. In the two end steps where the singleton changes column, the notation $j_{r;c\prime}$ explicitly overrides the side column and means $(x_{r,c\prime},B_j)$. These steps use respectively $\rho_{3,4}^{3,4}$ and $\rho_{4,13}^{1,2}$. A line break only continues the same chain.

\begin{table}[t]\centering\small\setlength{\tabcolsep}{7pt}\renewcommand{\arraystretch}{1.16}
\begin{tabular}{c l}\hline Side column $c$&Transfer chains\\\hline

$4$&$\begin{gathered}14_{3}\leadsto 5_{10}\leadsto 17_{15}\leadsto 20_{12}\leadsto 19_{14}\leadsto 21_{12}\leadsto 17_{14}\leadsto\\[-1pt]\phantom{\leadsto}\;21_{10}\leadsto 14_{14}\leadsto 20_{9}\leadsto 16_{15}\leadsto 18_{10}\leadsto 17_{11}\leadsto 10_{12}\leadsto\\[-1pt]\phantom{\leadsto}\;19_{6}\leadsto 6_{12}\end{gathered}$\\[4pt]
$4$&$\begin{gathered}19_{3}\leadsto 4_{12}\leadsto 17_{16}\leadsto 9_{10}\leadsto 14_{5}\leadsto 8_{9}\leadsto 16_{13}\leadsto\\[-1pt]\phantom{\leadsto}\;12_{10}\end{gathered}$\\[4pt]
$4$&$\begin{gathered}14_{6}\leadsto 10_{9}\leadsto 16_{11}\leadsto 15_{10}\leadsto 17_{9}\leadsto 16_{12}\leadsto 19_{10}\end{gathered}$\\[4pt]
$4$&$\begin{gathered}21_{5}\leadsto 9_{14}\leadsto 21_{16}\leadsto 4_{14}\leadsto 21_{3}\leadsto 5_{14}\leadsto 21_{15}\end{gathered}$\\[4pt]
$4$&$\begin{gathered}19_{15}\leadsto 18_{12}\leadsto 19_{11}\leadsto 15_{12}\leadsto 19_{9}\leadsto 14_{12}\end{gathered}$\\[4pt]
$4$&$\begin{gathered}15_{14}\leadsto 20_{11}\leadsto 10_{15}\leadsto 18_{6}\leadsto 6_{11}\end{gathered}$\\[4pt]
$4$&$\begin{gathered}21_{9}\leadsto 16_{14}\leadsto 20_{10}\leadsto 14_{15}\leadsto 18_{9}\end{gathered}$\\[4pt]
$4$&$\begin{gathered}5_{12}\leadsto 17_{3}\leadsto 4_{10}\leadsto 14_{16}\leadsto 2_{9}\end{gathered}$\\[4pt]
$2$&$\begin{gathered}7_{6}\leadsto 6_{8}\leadsto 13_{4}\leadsto 7_{7}\leadsto 12_{8}\end{gathered}$\\[4pt]
$4$&$\begin{gathered}8_{6}\leadsto 10_{5}\leadsto 9_{11}\leadsto 18_{16}\leadsto 2_{15}\end{gathered}$\\[4pt]
$4$&$\begin{gathered}15_{3}\leadsto 4_{9}\leadsto 16_{16}\leadsto 2_{10}\end{gathered}$\\[4pt]
$4$&$\begin{gathered}15_{5}\leadsto 8_{11}\leadsto 10_{13}\leadsto 12_{6}\end{gathered}$\\[4pt]
$4$&$\begin{gathered}16_{5}\leadsto 9_{9}\leadsto 15_{16}\leadsto 2_{11}\end{gathered}$\\[4pt]
$4$&$\begin{gathered}17_{5}\leadsto 9_{12}\leadsto 19_{16}\leadsto 2_{12}\end{gathered}$\\[4pt]
$4$&$\begin{gathered}18_{3}\leadsto 4_{11}\leadsto 10_{16}\leadsto 2_{6}\end{gathered}$\\[4pt]
$4$&$\begin{gathered}18_{5}\leadsto 8_{15}\leadsto 20_{13}\leadsto 12_{14}\end{gathered}$\\[4pt]
$4$&$\begin{gathered}20_{6}\leadsto 10_{14}\leadsto 21_{11}\leadsto 18_{14}\end{gathered}$\\[4pt]
$4$&$\begin{gathered}16_{3}\leadsto 5_{9}\leadsto 15_{15}\end{gathered}$\\[4pt]
$4$&$\begin{gathered}16_{6}\leadsto 10_{10}\leadsto 14_{11}\end{gathered}$\\[4pt]
$4$&$\begin{gathered}19_{5}\leadsto 8_{12}\leadsto 17_{13}\end{gathered}$\\[4pt]
$4$&$\begin{gathered}20_{5}\leadsto 9_{15}\leadsto 5_{16}\end{gathered}$\\[4pt]
$4$&$\begin{gathered}4_{6}\leadsto 6_{3}\leadsto 7_{3;3}\end{gathered}$\\[4pt]
$4$&$\begin{gathered}8_{10}\leadsto 14_{13}\leadsto 12_{9}\end{gathered}$\\[4pt]
\hline\end{tabular}\caption{Grounded one-edge--two-edge chains, part 1. The notation $j_r$ means $(x_{rc},B_j)$ for the side column in the first column.}\label{tab:seed-chains-1}\end{table}

\begin{table}[t]\centering\small\setlength{\tabcolsep}{7pt}\renewcommand{\arraystretch}{1.16}
\begin{tabular}{c l}\hline Side column $c$&Transfer chains\\\hline

$3$&$\begin{gathered}1_{2}\leadsto 3_{1}\end{gathered}$\\[4pt]
$3$&$\begin{gathered}1_{3}\leadsto 5_{1}\end{gathered}$\\[4pt]
$4$&$\begin{gathered}10_{3}\leadsto 5_{11}\end{gathered}$\\[4pt]
$2$&$\begin{gathered}11_{2}\leadsto 3_{8}\end{gathered}$\\[4pt]
$2$&$\begin{gathered}13_{2}\leadsto 3_{7}\end{gathered}$\\[4pt]
$4$&$\begin{gathered}17_{6}\leadsto 6_{10}\end{gathered}$\\[4pt]
$4$&$\begin{gathered}20_{16}\leadsto 4_{15}\end{gathered}$\\[4pt]
$4$&$\begin{gathered}21_{6}\leadsto 6_{14}\end{gathered}$\\[4pt]
$2$&$\begin{gathered}3_{4}\leadsto 8_{4;1}\end{gathered}$\\[4pt]
$4$&$\begin{gathered}5_{5}\leadsto 9_{3}\end{gathered}$\\[4pt]
$4$&$\begin{gathered}5_{6}\leadsto 6_{15}\end{gathered}$\\[4pt]
$2$&$\begin{gathered}6_{2}\leadsto 2_{6}\end{gathered}$\\[4pt]
$2$&$\begin{gathered}6_{7}\leadsto 11_{4}\end{gathered}$\\[4pt]
$2$&$\begin{gathered}7_{2}\leadsto 2_{8}\end{gathered}$\\[4pt]
$4$&$\begin{gathered}8_{14}\leadsto 21_{13}\end{gathered}$\\[4pt]
$4$&$\begin{gathered}8_{16}\leadsto 2_{5}\end{gathered}$\\[4pt]
$4$&$\begin{gathered}8_{3}\leadsto 5_{13}\end{gathered}$\\[4pt]
$4$&$\begin{gathered}9_{6}\leadsto 6_{16}\end{gathered}$\\[4pt]
\hline\end{tabular}\caption{Grounded one-edge--two-edge chains, part 2. The notation $j_r$ means $(x_{rc},B_j)$ for the side column in the first column.}\label{tab:seed-chains-2}\end{table}

These three sets of one-edge--two-edge targets are disjoint and exhaustive: $232+120+110=462$. Their coverage can be checked by listing the twenty-two black-dot positions and the twenty-one pairs in the seed; every non-line target is explicitly in Table~\ref{tab:seed-immediate} or occurs in one of the chains. All arrows terminate at previously grounded values, so there is no use of an ungrounded cycle.

\subsection{Completion of the seed proof}
Take two distinct one-edges. A common row or column gives orthogonality immediately. Otherwise their rectangle has at least one companion cell belonging to a two-edge, since four one-edges would form a forbidden $C_4$. The companion pair is consequently either a one-edge--two-edge pair, two distinct two-edges, or the two halves of one resolved two-edge. Its correct prescribed value has already been proved in all three cases, and one rectangle transfer certifies the target. Thus every one-edge pair is orthogonal as well.

We have proved all $\binom{43}{2}=903$ obligations from explicit rules. Together with simplicity, $|E_1|=22=z(16,4)$, and the one-edge skeleton visible in the grid, this proves $K_0$ is admissible and extremal. This finite base and the arbitrary-index arguments of Section~\ref{sec:general} provide a self-contained mathematical proof of Theorem~\ref{thm:main}.

\end{document}